\documentclass [11pt,a4paper]{article}
\usepackage[T1]{fontenc}
\usepackage[ansinew]{inputenc}
\usepackage{amssymb}
\usepackage{amsmath}
\usepackage{graphicx}
\usepackage{amsthm}
\usepackage{cite}
\newtheorem{theorem}{\textbf{Theorem}}[section]
\newtheorem{lemma}{\textbf{Lemma}}[section]
\newtheorem{proposition}{\textbf{Proposition}}[section]
\newtheorem{corollary}{\textbf{Corollary}}[section]
\newtheorem{remark}{\textbf{Remark}}[section]
\newtheorem{definition}{\textbf{Definition}}[section]

\allowdisplaybreaks[4]

\def\be{\begin{equation}}
\def\ee{\end{equation}}
\def\bea{\begin{eqnarray}}
\def\eea{\end{eqnarray}}
\def\bt{\begin{theorem}}
\def\et{\end{theorem}}
\def\bl{\begin{lemma}}
\def\el{\end{lemma}}
\def\br{\begin{remark}}
\def\er{\end{remark}}
\def\bp{\begin{proposition}}
\def\ep{\end{proposition}}
\def\bc{\begin{corollary}}
\def\ec{\end{corollary}}
\def\bd{\begin{definition}}
\def\ed{\end{definition}}

\def\s0t{\sup _{0\leq \tau\leq t}}
\def\C0T{C([0,T];\,}

\def\DAS{D( A^{\frac{S}{2}})}
\def\DAS1{D( A^{ \frac{S+1}{2}})}

 \def\non{\nonumber }

\def \au {\rm}

\def \no#1#2#3 {{\bf #1} (#3), #2.}
\def \eds#1#2#3 {#1, #2, #3.}

\begin{document}

\title{Asymptotic Behavior for a Nematic Liquid Crystal Model
with Different Kinematic Transport Properties}

\author{{\sc
Hao Wu} \footnote{Shanghai Key Laboratory for Contemporary Applied Mathematics and School of Mathematical Sciences, Fudan University,
200433 Shanghai, China, Email: \textit{haowufd@yahoo.com}.}, {\sc
Xiang Xu} \footnote{Department of Mathematics, Penn State
University, State College, PA 16802, Email:
\textit{xu\_x@math.psu.edu}.} \ and {\sc Chun Liu}
\footnote{Department of Mathematics, Penn State University, State
College, PA 16802, Email: \textit{liu@math.psu.edu}.} }

\date{\today}

\maketitle


\begin{abstract}
We study the asymptotic behavior of global solutions
to hydrodynamical systems modeling the nematic liquid crystal flows
under kinematic transports for molecules of  different shapes. The coupling system consists
of Navier--Stokes equations and kinematic
transport equations for the molecular orientations. We prove the
convergence of global strong solutions to single steady states as time
tends to infinity as well as estimates on the convergence rate both in
$2D$ for arbitrary regular initial data and in $3D$ for certain
particular cases.
\medskip

\noindent \textbf{Keywords}: Liquid crystal flows, Navier--Stokes
equation, kinematic transport, uniqueness of asymptotic limit, \L ojasiewicz--Simon approach. \\
\textbf{AMS Subject Classification}: 35B40, 35B41, 35Q35, 76D05.
\end{abstract}

\section{Introduction}
We consider the following evolutionary system that models the dynamics of
nematic liquid crystal flows (cf. e.g., \cite{LLZ07, SL08, dG, E61})
 \bea
 && v_t+v\cdot\nabla v-\nu \Delta v+\nabla P\non\\
 && \ \ =-\lambda
 \nabla\cdot[\nabla d\odot\nabla d+\alpha(\Delta d-f(d))\otimes d-(1-\alpha)d\otimes (\Delta d-f(d)) ],\label{1}\\
 && \nabla \cdot v = 0,\label{2}\\
 && d_t+(v\cdot\nabla) d-\alpha  (\nabla v) d+(1-\alpha)(\nabla^T v) d =\gamma(\Delta d-f(d)),\label{3}
 \eea
in $Q\times(0,\infty)$. Here, $Q$ is a unit square in
$\mathbb{R}^2,\ \!$ (or a unit box in $\mathbb{R}^3$), the more general case $Q=\Pi_{i=1}^n(0,L_i)$, $n=2,3$, with different periods $L_i$ in different directions can be treated in a similar way. $v$
is the velocity field of the flow and $d$ represents the averaged
macroscopic/continuum molecular orientations in $\mathbb{R}^n$. $P$
is a scalar function representing the hydrodynamic pressure. The
constants $\nu, \lambda$ and $\gamma$ stand for viscosity, the
competition between kinetic energy and potential energy, and
macroscopic elastic relaxation time (Deborah number) for the
molecular orientation field, respectively. The constant $\alpha\in
[0,1]$ is a parameter related to the shape of the liquid crystal
molecule (cf. Remark \ref{shape}). $\nabla d\odot \nabla d$ denotes the $n\times n$ matrix
whose $(i,j)$-th entry is given by $\nabla_i d\cdot \nabla_j d$, for
$1\leq i,j\leq n$. $\otimes$ is the usual Kronecker product, e.g.,
$(a\otimes b)_{ij}=a_ib_j$ for $a,b \in \mathbb{R}^n$.
The penalty function $f(d)=\frac{1}{\eta^2}(|d|^2-1)d: \mathbb{R}^n\mapsto \mathbb{R}^n$
with $\eta\in (0, 1]$ was introduced to approximate the strict unit-length constraint $|d| = 1$, which is
due to liquid crystal molecules being of similar size (cf.
\cite{LL95}). This approximation fits well with the general theory
of Landau's order parameter (cf. \cite{Le79}) and the
Ginzburg--Landau type energy is also consistent with the model on
variable degree of orientation (cf. \cite{E91}). It is obvious that
$f(d)$ is the gradient of the scalar valued function
$F(d)=\frac{1}{4\eta^2}(|d|^2-1)^2:\mathbb{R}^n\mapsto \mathbb{R}$.

The hydrodynamic theory of liquid crystals due to Ericksen and
Leslie was developed during the period of 1958 through 1968 (cf.
e.g., \cite{E61,Le68}). Since then there has been a remarkable
research in the field of liquid crystals, both theoretically and
experimentally (cf.
\cite{CS,dG,E87,E91,HK87,Le79,LD,LL95,LL96,LL01,LLZ07,LWa01,LSY07,SL08,LW08}
and references therein). System \eqref{1}--\eqref{3} is a simplified
version of the general Ericksen--Leslie system for the hydrodynamics
of nematic liquid crystals (cf. \cite{E61,E87,HK87,Le68,Le79,SL08}).
It is a macroscopic continuum description of time evolutions
of these materials influenced by both the flow field $v(x, t)$, and
the microscopic orientational configuration $d(x, t)$, which can be
derived from the coarse graining of the directions of liquid crystal
molecules. Equation \eqref{1} is the conservation of linear momentum
(the force balance equation). It combines a usual equation
describing the flow of an isotropic fluid with an extra nonlinear
coupling term that is anisotropic. This extra term is the induced
elastic stress from the elastic energy through the transport, which
is represented by the equation for $d$. Equation \eqref{2} simply
represents incompressibility of the fluid. Equation \eqref{3} is
associated with conservation of the angular momentum. The left-hand
side of \eqref{3} stands for the kinematic transport by the flow
field, while the right-hand side denotes the internal relaxation due
to the elastic energy (cf. e.g., \cite{SL08}).

The above system was derived from the macroscopic point of view and
was very successful in understanding the coupling between the
velocity field $v$ and the director field $d$, especially in the
liquid crystals of nematic type. In many experiments and earlier
theories on nematic liquid crystals, the samples are treated as
consisting of slow moving particles. Hence, one approach is to study
the behavior of the director field in the absence of the velocity
field. Unfortunately, the flow velocity does disturb the alignment
of the molecules. Moreover, the converse is also true: a change in
the alignment of molecules will induce velocity and this velocity
will in turn affect the time evolution of the director field. In
this process, we cannot assume that the velocity field will remain
small even if we start from zero velocity.

We recall that in the context of hydrodynamics, the basic variable
is the flow map (particle trajectory) $x(X, t)$. $X$ is the original
labeling of the particle (the Lagrangian coordinate), which is also
referred to as the material coordinate. $x$ is the current
(Eulerian) coordinate, and is also called the reference coordinate.
For a given velocity field $v(x, t)$, the flow map is defined by the
following ODEs:
\[\dot x=v(x(X, t), t), \;\;\  \  x(X, 0)=X. \]
To incorporate elastic properties of the material, we introduce the
deformation tensor \[\mathcal{F}(X, t)=\frac{\partial x}{\partial
X}(X,t). \] This quantity is defined in the Lagrangian material
coordinate and it satisfies $$ \frac{\partial
\mathcal{F}(X,t)}{\partial t}=\frac{\partial v(x(X,t),t)}{\partial
X}.
 $$
 In Eulerian coordinates, we define
$\tilde{\mathcal{F}}(x, t)=\mathcal{F}(X, t)$. By using the chain
rule, the above equation can be transformed into the following
transport equation for $\tilde{\mathcal{F}}$ (cf. e.g., \cite{GU}):
\begin{eqnarray*}
\tilde{\mathcal{F}}_t+(v\cdot\nabla)\tilde{\mathcal{F}}&=&\nabla
v\tilde{\mathcal{F}}.
\end{eqnarray*}
Without ambiguity, we will not distinguish the notations
$\mathcal{F}$ and $\tilde{\mathcal{F}}$ in the following text.

If the liquid crystal molecule has a rod-like shape, then transport
of the direction field $d$ can be expressed as (cf. \cite{Je, Lar})
\[ d(x(X, t), t)=\mathcal{F}d_0(X), \]
where $d_0(X)$ is the initial condition. This equation demonstrates
the stretching of the director besides the transport along the
trajectory. By taking full time derivative on both sides, we have
(cf. e.g., \cite{SL08})
\[\frac{D}{Dt}d(x(X, t), t)=\dot{\mathcal{F}}d_0(X)=\nabla v\mathcal{F}d_0
=\nabla v d=(d\cdot\nabla)v.
 \]
Hence, the total transport of the orientation vector $d$ becomes
\[d_t+v\cdot\nabla d-d\cdot\nabla v, \]
which represents the covariant parallel transport with no-slip
boundary condition between the rod-like particle and the fluid (cf.
\cite{Je, Lar}). In general, for molecules of ellipsoidal shape with
a finite aspect ratio, the transport of the main axis direction is
represented by
\[d(x(X, t), t)=\mathbb{E}d_0(X), \]
where $\mathbb{E}$ is a linear combination of $\mathcal{F}$ and
$\mathcal{F}^{-T}$ that satisfies the transport equation:
\[\mathbb{E}_t+(v\cdot\nabla)\mathbb{E}=
(\alpha\nabla v+(1-\alpha)(-\nabla^{T}v))\mathbb{E}.  \] As a
consequence, the total transport of $d$ in the general case
$\alpha\in[0,1]$ becomes
\[ d_t+v\cdot\nabla d-\omega d-(2\alpha-1)Ad, \] where $A=\frac{\nabla
v+\nabla^Tv}{2}$, $\omega=\frac{\nabla v-\nabla^Tv}{2}$.
\begin{remark} \label{shape}
We note that the spherical, rod-like and disc-like
liquid crystal molecules correspond to $\alpha=\frac{1}{2}$, $1$ and
$0$, respectively (cf. e.g., \cite{dG,E61,Je,SL08,LLZ07}).
\end{remark}

It is worth mentioning that system \eqref{1}--\eqref{3} can be
derived from an energetic variational point of view (cf.
\cite{HKL10}). The least action principle (LAP) with action
functional gives the Hamiltonian parts (conservative forces) of the
hydrodynamic system, and the maximum/minimum dissipation principle
(MDP), i.e., Onsager's principle, yields the dissipative parts
(dissipative forces). We refer to \cite{SL08} for a detailed
discussion particularly on the rod-like molecule system, i.e.,
$\alpha=1$ (see also \cite{HKL10} for illustrations on an immiscible
two-phase flow model as well as an incompressible viscoelastic
complex fluid model).

In the current paper, we consider system \eqref{1}--\eqref{3}
subject to the periodic boundary conditions (namely, $v,d$ are well
defined in the $n$-dimensional torus
$\mathbb{T}^n=\mathbb{R}^n/\mathbb{Z}$)
 \be
v(x+e_i)=v(x),\quad d(x+e_i)= d(x),\qquad \text{for}\ x\in
\mathbb{R}^n,
 \label{4}
 \ee
 where unit vectors $e_i \ (i=1,...,n)$ are the canonical basis of
 $\mathbb{R}^n$. Besides, we suppose the initial conditions
 \be
 v|_{t=0}=v_0(x) \ \ \text{with}\ \nabla\cdot v_0=0,\quad
 d|_{t=0}=d_0(x),\qquad \text{for}\ x\in Q.\label{5}
 \ee

As far as the related mathematical results are concerned,  we notice
that the current system \eqref{1}--\eqref{3} is a properly
simplified version of the general Ericksen--Leslie system for the
nematic liquid crystal flows (cf. \cite{CS, dG, LL01, Le79}).
However, the general Ericksen--Leslie system was so complicated that
only some special cases of it have been investigated theoretically
or numerically in the literature. A highly simplified system subject
to Dirichlet boundary conditions was first studied in \cite{LL95}.
Due to the dissipation of total energy, existence of global weak
solutions was obtained therein. Moreover, the authors proved
existence and uniqueness of global classical solutions in $2D$ as
well as some corresponding results in $3D$ (provided that the
constant viscosity $\nu$ was assumed to be large enough). Later on,
the authors proved in \cite{LL96} a partial regularity result that
the one dimensional space-time Hausdorff measure of the singular set
of suitable weak solutions to the system in \cite{LL95} was zero.
Numerical code using finite element methods was introduced in
\cite{LWa01, LSY07} to study the interaction of the defects and the
flow fields. Related to system \eqref{1}--\eqref{3}, a $C^0$ finite
element scheme preserving the energy law was established in
\cite{LLZ07} for simulating the kinematic effects in liquid crystal
dynamics. In \cite{SL08}, the authors studied global existence of
weak/classical solutions to system \eqref{1}--\eqref{5} for the
special case $\alpha=1$ (i.e., rod-like molecule) in $2D$ and $3D$
with large viscosity. For the general Ericksen--Leslie system,
results on the existence of solutions were obtained in \cite{LL01}
for a special case, i.e., when a maximum principle held for the
equation of $d$.

We note that in the highly simplified system studied in \cite{LL95},
the liquid crystal molecules are assumed to be "small" that no
kinematic transport was considered there. This somewhat lacks
significant physical meaning and is rather different from our
current problem \eqref{1}--\eqref{5}. However, it is quite
interesting that system \eqref{1}--\eqref{5} still enjoys a
(dissipative) \emph{basic energy law} similar to \cite{LL95}, which
governs the dynamics of the liquid crystal flow. To see this point,
we let $(v, d)$ be a classical solution to system
\eqref{1}--\eqref{5}. Multiplying equation \eqref{1} with $v$,
equation \eqref{3} with $-\lambda(\Delta d-f(d))$, adding them
together and integrating over $Q$, we get (cf. also \cite{LLZ07})
\[\frac{1}{2}\frac{d}{dt}\int_{Q}\left(|v|^2+\lambda|\nabla d|^2+2\lambda F(d)\right) dx
=-\int_{Q}\left(\nu|\nabla v|^2+\lambda\gamma|\Delta d-f(d)|^2
\right)dx. \]

Comparing with the small molecule system \cite{LL95}, we now have
different kinematic transports, which lead to two more stress terms
$\alpha(\Delta d-f(d))\otimes d-(1-\alpha)d\otimes(\Delta d-f(d))$
in the elastic stress in \eqref{1} and two more transport
terms $-\alpha (\nabla v)d+(1-\alpha)(\nabla ^T v) d$ in \eqref{3}.
These bring extra mathematical difficulties to prove the
well-posedness results of our system. For instance, among those
kinematic transport terms, there is an extra stretching effect on
the director field $d$, which leads to the loss of maximum principle
for the equation of $d$ (cf. \cite{LL95, LL01}). On the other hand,
the extra stress terms cannot be suitably defined in the weak
formulation as in \cite{LL95}, and thus the requirement that $d\in
L^\infty(0,T; L^\infty(Q))$ must be imposed to ensure the
well-posedness of the problem. (We refer to \cite{SL08} for a
discussion on the special case $\alpha=1$.) Here, we
have to confine ourselves to the periodic boundary conditions.
The corresponding initial boundary value problem  for
\eqref{1}--\eqref{3} (with e.g., Dirichlet boundary conditions) are still open,
because one cannot get rid of certain boundary
terms when performing integration by parts, which will bring extra
difficulties in the derivation of higher-order energy inequalities.

Since the parameters $\lambda$ and $\gamma$ do not
play important roles in the proof, we just set $\lambda=\gamma=1$ in
the remaining part of the paper. We now state the main results of this paper (see Section 2 for
functional settings).
 \bt
 \label{main2d}
 When $n=2$, for any initial data $(v_0, d_0)\in V\times H^2_p(Q)$ and $\alpha\in [0,1]$, problem
 \eqref{1}--\eqref{5} admits a unique
 global classical solution. It converges
 to a steady state $(0, d_\infty)$ as time goes to
 infinity such that
 \be \lim_{t\rightarrow +\infty}
 (\|v(t)\|_{H^1}+\|d(t)-d_\infty\|_{H^2})=0,\label{cgce}
 \ee
 where $d_\infty$ satisfies the following nonlinear
elliptic periodic boundary value problem:
  \be  - \Delta d_\infty + f(d_\infty)=0,\ \quad
  d_\infty(x+e_i)=d_\infty(x), \ \ x\in \ \!\mathbb{R}^n.
   \label{staa}
   \ee
 Moreover,
 \be
 \|v(t)\|_{H^1}+\|d(t)-d_\infty\|_{H^2}\leq C(1+t)^{-\frac{\theta}{(1-2\theta)}}, \quad \forall\ t \geq
 0.\label{rate}
 \ee
 In above,  $C$ is a positive constant depending on $\nu$, $\|v_0\|_{H^1}, \|d_0\|_{H^2}, \|d_\infty\|_{H^2}$. The constant $\theta \in (0,1/2)$ is the so-called \L ojasiewicz exponent depending on $d_\infty$ (cf. Lemma \ref{ls}).
 \et
 \bt\label{main3da}
 When $n=3$, for any initial data $(v_0, d_0)\in V\times H^2_p(Q)$ and $\alpha\in [0,1]$, if
 the constant viscosity is sufficiently large such that $\nu\geq
 \nu_0(v_0,d_0)$ (cf. \eqref{nu}), then problem
 \eqref{1}--\eqref{5} admits a unique global classical solution
 enjoying the same asymptotic behavior as in Theorem \ref{main2d}.
 \et
 \bt\label{main3d}
 When $n=3$, let $d^*\in H_p^2(Q)$ be an absolute minimizer of the functional
 \be E(d)=\frac12\|\nabla d\|^2+ \int_Q F(d)dx.
  \ee
There exists a constant $\sigma\in (0,1]$, which may depend on
$\nu$, $f$ and $d^*$, such that for $\alpha\in [0,1]$ and any initial data $(v_0, d_0)\in V\times H^2_p(Q)$
 satisfying
 $\|v_0\|_{H^1}+\|d_0-d^*\|_{H^2}<\sigma$, problem
 \eqref{1}--\eqref{5} admits a unique global classical solution
 enjoying the same asymptotic behavior as in Theorem \ref{main2d}.
 \et
  The problem whether a global bounded solution of nonlinear evolution equations will
converge to a single equilibrium as time tends to infinity, has
attracted a lot of interests among mathematicians for a long time.
It is well known that the structure of the set of equilibria can be
nontrivial and may form a continuum for certain physically
reasonable nonlinearities in higher dimensional case. In particular,
under current periodic boundary conditions, one may expect that
the dimension of the
 set of equilibria is at least $n$. This is  because a shift in each
 variable should give another steady state (cf. also \cite{RH99}),
 e.g., in our case, if $d^*$ is a steady state, so is $d^*(\cdot\ \!+\tau e_i)$, $1\leq i\leq
 n, \ \!\tau\in \mathbb{R}^+$. Moreover, we note that for our
 system, every constant vector $d$ with unit-length ($|d|=1$) serves as an absolute
 minimizer of the functional $E(d)$. As a result, it is
highly nontrivial to decide whether a given trajectory will converge
to a single equilibrium, or in other words, whether the
$\omega$-limit set is a single point. Fortunately, we are able to
apply the so-called \L ojasiewicz--Simon approach (cf. L. Simon
\cite{S83}) to obtain our goal. Simon's idea relies on a nontrivial
generalization of the
  \L ojasiewicz inequality (cf. \cite{L1,L2}) for analytic functions defined in the finite
  dimensional space $\mathbb{R}^m$ to infinite dimensional spaces.
  We refer to
\cite{J981,HJ01,ET01,WGZ1,WGZ4,LD,RH99,W07,Hu,Z04} and the
references therein for applications to various evolution equations.
  In order to apply the \L ojasiewicz--Simon approach to our
problem \eqref{1}--\eqref{5}, we need to introduce a suitable
{\L}ojasiewicz--Simon type inequality for vector functions with
periodic boundary condition (cf. Lemma \ref{ls} below). Although
different kinematic transports for the liquid crystal molecules will
yield different dynamics of the hydrodynamical system, we shall show
that global solutions to our system have some similar long-time
behavior under different kinematic transports with $\alpha\in [0,1]$
(i.e., convergence to equilibrium with a uniform convergence rate in
the parameter $\alpha$). Our results can be considered as a
nontrivial generalization of the result in \cite{LW08}, where the
corresponding asymptotic results for the small molecule system (cf.
\cite{LL95}) were discussed under various boundary conditions (e.g.,
Dirichlet b.c./free-slip b.c.).

 The remaining part of this paper is organized as follows. In Section 2,
 we introduce the functional setting and some preliminary results.
 Section 3 is devoted to the $2D$ case. We prove the convergence of
 global solutions to single steady states as time
tends to infinity for arbitrarily regular initial data and obtain an
estimate on convergence rate that is uniform in $\alpha$. In Section
4, we study the $3D$ case. The same convergence result is proved for
two subcases, in which global existence of classical solutions can
be obtained. In Section 5, we discuss the results for liquid crystal
flows with non-vanishing average velocity.

\section{Preliminaries}
\setcounter{equation}{0}
 We recall the well-established functional
setting for periodic boundary value problems (cf. e.g.,
\cite[Chapter 2]{Te}):
 \bea
H^m_p(Q)&=&\{v\in H^m(\mathbb{R}^n,\mathbb{R}^n)\ |\ v(x+e_i)=v(x)\},\non\\
 \dot{H}^m_p(Q) &=& H^m_p(Q)\cap \left\{v:\ \int_Qv(x)dx=0\
\right\},\non\\
H&=&\{v\in L^2_p(Q,\mathbb{R}^n) ,\ \nabla\cdot v=0\},\ \
\text{where}\ L^2_p(Q,\mathbb{R}^n)=H^0_p(Q),
\non\\
V&=&\{v\in \dot{H}^1_p(Q),\ \nabla\cdot v=0\},\non\\
V'&=&\text{the\ dual space of\ } V.\non
 \eea
 For the sake of simplicity, we denote the inner product on
$L^2_p(Q,\mathbb{R}^n)$ as well as  $H$ by $(\cdot,\cdot)$ and the
associated norm by $\|\cdot\|$. We shall denote by $C$ the genetic
constants depending on $\lambda, \gamma, \nu, Q, f$ and the initial
data. Special dependence will be pointed out explicitly in the text
if necessary.

Following \cite{Te}, one can define mapping $S$
 \be S u=-\Delta u, \quad  \forall\ u\in D(S):=\{u\in H, \Delta u\in H\}=\dot H^2_p\cap
 H.\label{stokes}
 \ee
 The operator $S$ can be viewed as an unbounded
positive linear self-adjoint operator on $H$. If $D(S)$ is endowed
with the norm induced by $\dot H^0_p(Q)$, then $S$ becomes an
isomorphism from $D(S)$ onto $H$. More detailed properties of
operator $S$ can be found in \cite{Te}.

We also recall
 the interior elliptic estimate, which states that for any $U_1\subset\subset U_2$
 there is a constant $C>0$ depending only on $U_1$ and $U_2$ such that
 $\|d\|_{H^2(U_1)}\leq C(\|\Delta d\|_{L^2(U_2)}+\|d\|_{L^2(U_2)})$. In our
 case, we can choose $Q'$ to be the union of $Q$ and its
 neighborhood copies. Then we have
 \be
 \|d\|_{H^2(Q)}\leq C(\|\Delta d\|_{L^2(Q')}+\|d\|_{L^2(Q')})= 9C(\|\Delta
 d\|_{L^2(Q)}+\|d\|_{L^2(Q)}).\label{dh2}
 \ee
The following
 embedding inequalities will be frequently used in the subsequent proofs:
 \bl (cf. \cite{Te}) If $n=2$, we have
  $$\|u\|_{L^\infty(Q)}\leq c\|u\|^\frac12\|u\|^\frac12_{H^2},\quad
  \forall\ u \in H^2_p(Q),$$
If $n=3$, then
$$\|u\|_{L^\infty(Q)}\leq c\|u\|_{H^1}^\frac12\|u\|_{H^2}^\frac12,\quad
  \forall\ u \in H^2_p(Q).$$
  Here, we note that $\|u\|_{H^2(Q)}$ can be estimated by $\|\Delta u\|$ and $\|u\|$
  in spirit of estimate \eqref{dh2}.
  \el
The global existence of weak/classical solutions to the system
\eqref{1}--\eqref{5} for $\alpha=1$ (i.e., the rod-like molecule)
has been proven in \cite[Theorem 1.1]{SL08}. The
proof relies on a modified Galerkin
 method introduced in \cite{LL95}. After
generating a sequence of approximate solutions, one applies the
Ladyzhenskaya method to get higher-order energy estimates (cf. Lemma
\ref{he2d}, Lemma \ref{he3da}), which enable us to pass to the
limit. Besides, a weak solution together with higher-order estimates
implies a strong solution. Furthermore, a bootstrap argument based
on Serrin's result \cite{Se} (cf. also \cite{K67}) and Sobolev
embedding theorems leads to the existence of classical solutions. It
is not difficult to extend the result in \cite{SL08} to our general
case $\alpha\in [0,1]$ and we shall omit the detailed proof here.
 \bt\label{glo}
For $\alpha\in [0,1]$ and any $(v_0, d_0)\in V\times H^2_p(Q)$, if either
 $n=2$ or $n=3$ with the large viscosity assumption $\nu\geq
 C(v_0,d_0)$ (cf. \eqref{nu}), problem \eqref{1}--\eqref{5} admits a global
 solution such that
 \be v\in L^\infty(0, \infty; V)\cap L^2_{loc}(0,\infty; D(S)),\quad d\in
 L^\infty(0,\infty; H^2)\cap L^2_{loc}(0,\infty; H^3),\label{so}
 \ee
 Moreover, $v,d\in C^\infty(Q)$ for all
 $t>0$.
 \et

As mentioned in the introduction, system \eqref{1}--\eqref{5} admits a
Lyapunov functional (recall that we have set $\lambda=\gamma=1$ for simplicity)
 \be \mathcal{E}(t)=\frac{1}{2}\|v(t)\|^2+\frac{1}{2}\|\nabla
 d(t)\|^2+\int_Q F(d(t))dx,\label{Ly}
 \ee
 which satisfies the following \textit{basic energy law}
 \be
 \frac{d}{dt}\mathcal{E}(t)=-\nu\|\nabla v(t)\|^2-\|\Delta
 d(t)-f(d(t))\|^2,
 \quad \forall \ t\geq 0.\label{ENL}
 \ee

Before ending this section, we shall show a continuous dependence
result on the initial data, from which one can infer the uniqueness
of regular solutions to problem \eqref{1}--\eqref{5}. We note that
no uniqueness result has been obtained before, even for the special
case $\alpha=1$ in \cite{SL08}.
 \bl \label{cd}
 Suppose that $(v_i,d_i)$ are global solutions to the problem
 \eqref{1}--\eqref{5} corresponding to initial data $(v_{0i}, d_{0i})\in V\times
 H^2_p(Q)$, $i=1,2$, respectively. Moreover, we assume that for any $T>0$, the following estimate
 holds
 \be
 \|v_i(t)\|_{H^1}+\|d_i(t)\|_{H^2}\leq M ,\quad \forall \
 t\in [0,T].\label{BD}
 \ee
 Then for any $ t\in [0,T]$, we have
\begin{eqnarray}
&& \| (v_1-v_2)(t)\|^2+\|(d_1-d_2)(t)\|_{H^1}^2 +
\int_0^t(\nu\|\nabla
 (v_1-v_2)(\tau)\|^2+\|\Delta( d_1-d_2)(\tau)\|^2) d\tau \non\\
 &\leq&  2e^{Ct}(\|v_{01}-v_{02}\|^2+\|d_{01}-d_{02}\|_{H^1}^2),
\end{eqnarray}
where $C$ is a constant depending on $M$ but not on $t$.
 \el
\begin{proof}
 Denote
\begin{eqnarray}
\bar v=v_1-v_2, \ \  \bar d=d_1-d_2.
\end{eqnarray}
Since $(v_i,d_i)$ are solutions to problem \eqref{1}--\eqref{5}, we
have
 \bea
 && v_{1t}+v_1\cdot\nabla v_1-\nu \Delta v_1+\nabla P_1\non\\
 &&\ =-
 \nabla\cdot[\nabla d_1\odot\nabla d_1+\alpha(\Delta d_1-f(d_1))\otimes d_1-(1-\alpha)d_1\otimes (\Delta d_1-f(d_1))],\label{1 for the first}\\
 && \nabla \cdot v_1 = 0,\label{2 for the first}\\
 && d_{1t}+v_1\cdot\nabla d_1-\alpha  (\nabla v_1) d_1+(1-\alpha)(\nabla^T v_1) d_1=\Delta d_1-f(d_1),\label{3 for the first}
 \\
 && v_{2t}+v_2\cdot\nabla v_2-\nu \Delta v_2+\nabla P_2\non\\
 &&\ =-\nabla\cdot[\nabla d_2\odot\nabla d_2+\alpha(\Delta d_2-f(d_2))\otimes d_2-(1-\alpha)d_2\otimes (\Delta d_2-f(d_2))],\label{1 for the second}\\
 && \nabla \cdot v_2 = 0,\label{2 for the second}\\
 && d_{2t}+v_2\cdot\nabla d_2-\alpha  (\nabla v_2) d_2+(1-\alpha)(\nabla^T v_2) d_2=\Delta d_2-f(d_2).\label{3 for the second}
 \eea
  Multiplying
$\bar v$ with the subtraction of \eqref{1 for the second} from
\eqref{1 for the first} and $ \bar d-\Delta \bar d$ with the
subtraction of \eqref{3 for the second} from \eqref{3 for the
first}, respectively, adding the two resultants together, using
integration by parts, we infer from the periodic boundary conditions
that
\begin{eqnarray}
&&\frac{1}{2}\frac{d}{dt}\left(\|\bar v\|^2+\|\bar d\|^2+\|\nabla
\bar d\|^2
\right)+\nu\|\nabla \bar v\|^2+\|\nabla \bar d\|^2+\|\Delta \bar d\|^2\nonumber\\
&=&-(v_2\cdot\nabla \bar v, \bar v)-(\bar v\cdot \nabla v_1, \bar
v)-(\Delta d_2\cdot \nabla \bar
d, \bar v)\non\\
&& +\alpha(\Delta d_2 \otimes \bar d, \nabla \bar
v)-\alpha((f(d_1)-f(d_2))\otimes d_1, \nabla \bar
v)-\alpha(f(d_2)\otimes \bar
d,\nabla \bar v)\non\\
&& -(1-\alpha)(\bar d \otimes \Delta d_2, \nabla \bar
v)+(1-\alpha)(d_1 \otimes (f(d_1)-f(d_2)), \nabla \bar
v)+(1-\alpha)(\bar
d \otimes f(d_2),\nabla \bar v)\non\\
&& +(f(d_1)-f(d_2), \Delta \bar d) +(v_2\cdot\nabla \bar d, \Delta
\bar d)-\alpha((\nabla v_2)\bar d, \Delta\bar d)+(1-\alpha)
((\nabla^T v_2)\bar d, \Delta\bar
d) \non\\
&& -(f(d_1)-f(d_2), \bar d)-(\bar v\cdot \nabla d_1,\bar d)-(v_2\cdot \nabla \bar d, \bar d)\non\\
&& +\alpha((\nabla \bar v) d_1,\bar d)+\alpha((\nabla v_2)\bar d,
\bar d)-(1-\alpha)((\nabla^T \bar v) d_1,\bar
d)-(1-\alpha)((\nabla^T v_2)\bar d, \bar d). \label{eni}
\end{eqnarray}
Assumption \eqref{BD} implies that $\|v\|_{H^1}$ and $\|d\|_{H^2}$
are uniformly bounded in $[0,T]$. Hence, by using the Sobolev
embedding theorems, we can estimate the right-hand side of
\eqref{eni} term by term (the calculation presented here is for the
$3D$ case and it is also valid for $2D$).
\begin{eqnarray}
&& |(v_2\cdot\nabla \bar v, \bar v)| +|(\bar v\cdot \nabla v_1, \bar
v)| \non\\
&\leq& \|v_2\|_{L^6}\|\nabla \bar v\|\|\bar v\|_{L^3}+\|\bar
v\|_{L^4}^2\|\nabla v_1\|\non\\
&\leq&  C\|\nabla \bar v\|(\|\nabla \bar v\|^\frac12\|\bar
v\|^\frac12+\|\bar v\|)+C(\|\nabla \bar v\|^\frac34\|\bar
v\|^\frac14+\|\bar
v\|)^2\non\\
&\leq& \varepsilon \|\nabla \bar v\|^2+ C\|\bar v\|^2.
\end{eqnarray}
\begin{eqnarray}
&&|(\Delta d_2\cdot \nabla \bar d, \bar v)|+|\alpha(\Delta d_2
\otimes \bar d, \nabla \bar v)|+|(1-\alpha)(\bar d \otimes \Delta
d_2, \nabla \bar v)|
\nonumber\\
&\leq& \|\Delta d_2\|\|\nabla \bar d\|_{L^3}\|\bar
v\|_{L^6}+\|\Delta d_2\|\|\bar d\|_{L^\infty}\|\nabla \bar v\|\non\\
&\leq& C(\|\Delta \bar d\|^\frac12\|\nabla \bar d\|^\frac12+\|\nabla
\bar d\|)(\|\nabla \bar v\|+\|\bar v\|)+ C(\|\Delta \bar
d\|^\frac34\|\bar d\|^\frac14+\| \bar d\|)\|\nabla \bar
v\|\non\\
&\leq & \varepsilon(\|\Delta \bar d\|^2+\|\nabla \bar v\|^2)+
C(\|\bar d\|_{H^1}^2+\|\bar v\|^2).
\end{eqnarray}
\begin{eqnarray}
&&|\alpha ((f(d_1)-f(d_2))\otimes d_1, \nabla \bar
v)|+|\alpha(f(d_2)\otimes \bar
d,\nabla \bar v)|\non\\
&& +|(1-\alpha)(d_1 \otimes (f(d_1)-f(d_2)), \nabla \bar
v)|+|(1-\alpha)(\bar
d \otimes f(d_2),\nabla \bar v)|\non\\
&& +|(f(d_1)-f(d_2), \Delta \bar d)| +|(f(d_1)-f(d_2), \bar d)| \nonumber\\
&\leq& (\|f(d_1)-f(d_2)\|\|d_1\|_{L^\infty}+
\|f(d_2)\|_{L^\infty}\|\bar d\|)\|\nabla \bar v\|
+\|f(d_1)-f(d_2)\|(\|\Delta \bar d\|+\|\bar d\|) \non\\
&\leq& C(\|f'(\xi)\|_{L^\infty}+1)\|\bar d\|\|\nabla \bar
v\|+C\|f'(\xi)\|_{L^\infty}\|\bar d\|(\|\Delta \bar d\|+\|\bar d\|)\non\\
&\leq& \varepsilon( \|\nabla \bar v\|^2+\|\Delta \bar d\|^2)+
C\|\bar d\|^2,
\end{eqnarray}
where $\xi=ad_1+(1-a)d_2$ with $a\in (0,1)$.
\begin{eqnarray}
&& |(v_2\cdot\nabla \bar d, \Delta \bar d)|+|\alpha((\nabla v_2)\bar
d, \Delta\bar d)| + |(1-\alpha) ((\nabla^T v_2)\bar d,
\Delta\bar d)|\non\\
&\leq& \|v_2\|_{L^6}\|\nabla \bar d\|_{L^3}\|\Delta \bar
d\|+\|\nabla v_2\|\|\bar d\|_{L^\infty}\|\Delta \bar d\|\non\\
&\leq& C(\|\Delta\bar d\|^\frac12\|\nabla \bar d\|^\frac12+\|\nabla
\bar d\|)\|\Delta \bar d\|+ C(\|\Delta \bar d\|^\frac34\|\bar
d\|^\frac14+\|\bar d\|)\|\Delta \bar d\|\non\\
&\leq& \varepsilon\|\Delta\bar d\|^2+C\|\bar d\|_{H^1}^2.
\end{eqnarray}
\begin{eqnarray}
&& |(\bar v\cdot \nabla d_1,\bar d)|+|(v_2\cdot \nabla \bar d, \bar
d)|+|\alpha((\nabla \bar v)d_1,\bar d)|+|\alpha((\nabla v_2)\bar d,
\bar
d)| \non\\
&& +|(1-\alpha)((\nabla^T \bar v) d_1,\bar d)|+|(1-\alpha)((\nabla^T
v_2)\bar d, \bar d)|\non\\
 &\leq & \|\nabla d_1\|_{L^3}\|\bar
v\|\|\bar d\|_{L^6}+ \|v_2\|_{L^6}\|\nabla \bar d\|\|\bar
d\|_{L^3}+\|d_1\|_{L^\infty}\|\nabla \bar v\|\|\bar d\|+ \|\nabla
v_2\|\|\bar d\|^2_{L^4}\non\\
&\leq& \varepsilon\|\nabla \bar v\|^2+C(\|v\|^2+\|\bar d\|_{H^1}^2).
\end{eqnarray}
Choosing $\varepsilon$ small enough in the above estimates, we infer
from \eqref{eni} that
\begin{eqnarray}
&&\frac{d}{dt}(\|\bar v\|^2+\|\bar d\|_{H^1}^2) + \nu \|\nabla \bar
v\|^2 + \|\Delta \bar d\|^2\leq C(\|\bar v\|^2+\|\bar d\|_{H^1}^2),
\end{eqnarray}
where the constant $C$ depends on $\|v_i\|_{H^1}$, $\|d_i\|_{H^2}$,
but not on $t$. We also note that $C$ is uniform for all $\alpha\in
[0,1]$.

 By Gronwall's
inequality, we can see that for any $t\in [0,T]$,
\begin{eqnarray}
\|\bar v(t)\|^2+\|\bar d(t)\|_{H^1}^2 + \int_0^t(\nu\|\nabla \bar
v(\tau)\|^2+\|\Delta \bar d(\tau)\|^2) d\tau \leq 2e^{Ct}(\|\bar
v(0)\|^2+\|\bar d(0)\|_{H^1}^2).
\end{eqnarray}
The proof is complete.
\end{proof}

 \bc  \label{uniq} The global
 solution $(v,d)$ obtained in Theorem \ref{glo} is unique.
 \ec
 \begin{proof}
 Since the global classical solution $(v,d)$ to the problem
 \eqref{1}--\eqref{5} obtained in Theorem \ref{glo} is uniformly bounded
 in $V\times H^2$ (cf. Lemma \ref{vcon} and Proposition \ref{AA1}),
 it follows immediately from Lemma \ref{cd} that the solution is
 unique.
 \end{proof}

\section{Results in Two Dimensional Case}
\setcounter{equation}{0}

\subsection{Higher-order estimates}
It follows from the basic energy law \eqref{ENL} that
 \be \mathcal{E}(t)+ \int_0^t (\nu\|\nabla v(\tau)\|^2+\|\Delta
 d(\tau)-f(d(\tau))\|^2) d\tau \leq
\mathcal{E}(0)<+\infty,\quad \forall\ t\geq 0.\label{ae}
 \ee
  From the definition of $\mathcal{E}(t)$ and \eqref{ae}, we conclude that
 \bl \label{v0h1}
  \be \|v(t)\|+\|d(t)\|_{H^1}\leq C,\quad  \forall\ t\geq 0,\label{bd1}
  \ee
  and
  \be \int_0^{+\infty} (\nu\|\nabla v(t)\|^2+\lambda\gamma\|\Delta
 d(t)-f(d(t))\|^2) dt\leq C,\label{intA}
  \ee
  where $C$ is a constant depends only on $\|v_0\|$, $\|d_0\|_{H^1}$.
  \el

Denote
 \be
 A(t)=\|\nabla v(t)\|^2+ \lambda\|\Delta d(t)-f(d(t))\|^2.\label{A}
 \ee
In $2D$ case, an important property for global solutions to
problem \eqref{1}--\eqref{5} is the following higher-order energy
inequality.
 \bl \label{he2d} In $2D$ case, the following inequality holds for the
 classical solution $(v,d)$ to problem \eqref{1}--\eqref{5}:
 \be \frac{d}{dt}A(t)+\nu\|\Delta v(t)\|^2+\|\nabla(\Delta d(t)-f(d(t)))\|^2\leq C(A^2(t)+A(t)), \quad \forall
 \
 t\geq 0,\label{he}
 \ee
 where $C$ is a constant depending on $\|v_0\|, \|d_0 \|_{H^1}, \nu, f, Q$, but independent of $\alpha$.
 \el
 \begin{proof}
A direct calculation yields that
 \bea
 && \frac12\frac{d}{dt}A(t)+\nu \|\Delta
v\|^2+\|\nabla(\Delta d-f(d))\|^2
\non\\
 &=&(\Delta v, v\cdot\nabla v)+(\Delta
v, \nabla\cdot(\nabla d \odot \nabla d))+\alpha(\Delta v, \nabla\cdot((\Delta d-f) \otimes d))
 \non \\&&-(1-\alpha)(\Delta v, \nabla\cdot(d \otimes (\Delta d-f) ))
 +(\nabla(\Delta d-f), \nabla(v\cdot\nabla d)) \non\\
 &&-\alpha(\nabla(d\cdot\nabla v), \nabla(\Delta d-f))+(1-\alpha)(\nabla(d\cdot\nabla^T v), \nabla(\Delta d-f))
 \non\\
 && -(\Delta d-f, f'(d)d_t)\non\\
 &:=& I_1+...+I_8.\label{diffA}
\eea
 Based on Lemma \ref{v0h1}, we can
estimate the right-hand side of the above equality term by term.
 \be
|I_1| \leq \|\Delta v\|\|v\|_{L^\infty}\|\nabla
v\| \leq C\|\Delta v\|^{\frac32}\|\nabla v\| \leq \frac{\nu}{32}\|\Delta
v\|^2+C\|\nabla v\|^4.
 \ee
 Recalling the definition of operator $S$ (cf. \eqref{stokes}), we see
 that for $v\in D(S)$, $\Delta v$ is also divergence free. Then
 \be I_2=\left(\Delta v,\nabla \left(\frac{|\nabla d|^2}{2}\right)\right)+ (\Delta v, \Delta d\cdot\nabla d)
= (\Delta v, \nabla d\cdot\Delta d).
\label{cancellation 5}
 \ee
Using integration by parts, we get
 \bea I_3 +I_6 &=& \alpha\left(\Delta v,
\nabla\cdot((\Delta d-f) \otimes
d)\right)-\alpha(\nabla(d\cdot\nabla v), \nabla(\Delta d-f))
\non\\
&=& -\alpha(d\cdot\nabla\Delta
v,\Delta
d-f)+\alpha(\Delta(d\cdot\nabla v), \Delta d-f)
\non\\
&=& \alpha(\Delta
d\cdot\nabla v, \Delta d-f)+2\alpha((\nabla d\cdot\nabla)\cdot\nabla
v, \Delta d-f)\non\\
&:=& I_3'+I_6' \label{cancellation 1}
 \eea
 \bea
I_4 +I_7&=& -(1-\alpha)(\Delta v, \nabla\cdot(d \otimes (\Delta d-f)
))-(1-\alpha)(\Delta(d\cdot\nabla^T v), \Delta d-f)\non\\
&=&-(1-\alpha)(\Delta d\cdot\nabla^T v, \Delta d-f) -2(1-\alpha)((\nabla d\cdot\nabla)\cdot\nabla^T v, \Delta d-f)
\non\\
&:=& I_4'+I_7' .
\label{cancellation 2}
 \eea
 Besides,
  \bea
I_5&=&(\nabla(\Delta d-f), \nabla(v\cdot\nabla d)) \non\\
&=&-(\Delta d-f, \Delta v\cdot\nabla d)-2(\Delta
d-f, \nabla v\cdot\nabla^2 d)-(\Delta d-f,
v\cdot\nabla\Delta d)\non\\
&:=& I_{5a}+I_{5b}+I_{5c}.\non
 \eea
It follows that
 \bea
I_{5a} &=& -(\Delta
v, \nabla d\cdot\Delta d)+  (\Delta v, \nabla F(d))= -(\Delta v, \nabla d\cdot\Delta d), \label{cancellation 6}        \\
I_{5c}&=& -(\Delta
d-f, v\cdot\nabla(\Delta
d-f))-(\Delta d-f, v\cdot\nabla f) =-(\Delta d-f, v\cdot\nabla f).
\label{cancellation 7}
 \eea
Then we see that
 $$I_2+I_{5a}=0.$$
 and
 \bea
 I_{5c}+I_8&=& -(\Delta d-f, v\cdot\nabla f) -(\Delta d-f, f'(d)d_t) \non\\
 &=& -(\Delta d-f, f'(d)(\Delta d-f)) +(\Delta d-f, f'(d)(-\alpha d\cdot\nabla
v+(1-\alpha)d\cdot\nabla^Tv))\non\\
&:=& I_{5c}'+I_8'.
 \eea
Due to the above cancelations, it remains to estimate the terms $I_3', I_4', I_{5b}', I_6', I_7', I_8'$.
By Agmon's inequality $(n=2)$,
 \be
\|d\|_{L^\infty} \leq C(1+\|\Delta d\|^{\frac12}), \ \ \ \ \
\|\nabla d\|_{L^\infty} \leq C(1+\|\nabla\Delta d\|^{\frac12}),
\label{agmon} \ee
where $C$ depends at most on $\|d\|_{H^1}$. Besides, we have
 \bea
 \|d\|_{H^2}&\leq& C(1+\|\Delta d\|)\leq C(1+\|\Delta d-f(d)\|+\|f(d)\|)\non\\
 &\leq& C\|\Delta d-f(d)\|+C(\|d\|_{L^6}^3+\|d\|)\leq C\|\Delta d-f(d)\|+C.\non
 \eea
 \bea \|\nabla \Delta d\|
 &\leq&  \|\nabla(\Delta d-f(d)\|+\|f'(d)\|_{L^\infty}\|\nabla
 d\| \leq \|\nabla(\Delta d-f(d)\|+ C(1+\|d\|_{L^\infty}^2)
 \non\\
 &\leq& \|\nabla(\Delta d-f(d)\|+C(1+\|\nabla\Delta d\|)\non\\
 &\leq& \|\nabla(\Delta d-f(d)\|+C(1+ \|\nabla \Delta
 d\|^\frac12\|\nabla d\|^\frac12+\|\nabla d\|)\non
 \\
 &\leq & \|\nabla(\Delta d-f(d)\| + \frac12 \|\nabla \Delta
 d\|+C, \non \eea
 which implies that
 \be \|\nabla \Delta d\|\leq \|\nabla(\Delta d-f(d)\| +C,
 \ee
 where $C$ depends at most on $\|d\|_{H^1}$.

 Then we have
 \bea |I_3'|&=& |\alpha((\Delta d-f)\cdot\nabla v, \Delta d-f)+\alpha(f\cdot\nabla v, \Delta
 d-f)|  \non\\
&\leq& \alpha\|\nabla v\|\|\Delta d-f\|^2_{L^4}+\alpha\|f\|\|\nabla
v\|_{L^4}\|\Delta d-f\|_{L^4}     \non\\
&\leq& C\|\nabla v\|(\|\Delta d-f\|\|\nabla(\Delta d-f)\|+\|\Delta
d-f\|^2)\non\\
&& +C\|\Delta v\|^\frac12\|\nabla v\|^\frac12(\|\Delta d-f\|^{\frac12}\|\nabla(\Delta
d-f)\|^{\frac12}+\|\Delta d-f\|)    \non\\
&\leq& \frac{\nu}{32}\|\Delta v\|^2+\frac{1}{32}\|\nabla(\Delta
d-f)\|^2\non\\
&& +C(\|\Delta d-f\|^4+\|\nabla v\|^4+\|\Delta d-f\|^2+\|\nabla
v\|^2).
 \label{similar 1}
  \eea
  \bea |I_6'| &\leq&
2\alpha\|\Delta d-f\|\|\nabla d\|_{L^\infty}\|D^2 v\| \non\\
&\leq& C\|\Delta d-f\|(\|\nabla(\Delta d-f)\|^\frac12+1)\|\Delta
v\| \non\\
&\leq& \frac{\nu}{32}\|\Delta v\|^2+\frac{1}{32}\|\nabla(\Delta
d-f)\|^2+C(\|\Delta d-f\|^4+\|\Delta d-f\|^2).
 \label{similar 2}
  \eea
 In a similar way, we obtain
 \bea
|I_4'|&\leq& \frac{\nu}{32}\|\Delta
v\|^2+\frac{1}{32}\|\nabla(\Delta d-f)\|^2\non\\
&& +C(\|\Delta d-f\|^4+\|\nabla v\|^4+\|\Delta
d-f\|^2+\|\nabla v\|^2). \\
|I_7'|&\leq& \frac{\nu}{32}\|\Delta v\|^2+\frac{1}{32}\|\nabla(\Delta
d-f)\|^2+C(\|\Delta d-f\|^4+\|\Delta d-f\|^2).
\eea
 Next,
 \bea |I_{5b}| &\leq&
2\|\Delta d-f\|_{L^4}\|\nabla v\|_{L^4}\|\nabla^2 d\|  \non\\
&\leq&  C(\|\nabla(\Delta d-f)\|^\frac12\|\Delta d-f\|^\frac12+\|\Delta d-f\|)
\|\Delta v\|^\frac12\|\nabla v\|^\frac12(\|\Delta
d-f\|+1)  \non\\
&\leq& C\|\nabla(\Delta d-f)\|^{\frac12}\|\Delta v\|^\frac12\|\Delta
d-f\|^\frac32\|\nabla v\|^\frac12+C\|\nabla(\Delta d-f)\|^{\frac12}\|\Delta v\|^\frac12\|\Delta
d-f\|^\frac12\|\nabla v\|^\frac12\non\\
&& + C\|\Delta v\|^\frac12\|\Delta
d-f\|^2\|\nabla v\|^\frac12+ C\|\Delta v\|^\frac12\|\Delta
d-f\|\|\nabla v\|^\frac12 \non\\
&\leq& \frac{1}{32}\|\nabla(\Delta d-f)\|^2+ \frac{\nu}{32}\|\Delta v\|^2+
C(\|\nabla v\|^4+\|\nabla
v\|^2+\|\Delta
d-f\|^4+\|\Delta
d-f\|^2).
 \eea
\bea
 |I_{5c}'| &\leq& \|f'(d)\|_{L^\infty}\|\Delta d-f\|^2\leq C(\|d\|_{L^\infty}^2+1)\|\Delta d-f\|^2 \non\\
&\leq& C(\|\Delta d-f\|^3+\|\Delta d-f\|^2)
 \eea
 \bea |I_8'| &\leq& \|\Delta
d-f\|_{L^4}\|f'(d)\|_{L^4}\|d\|_{L^4}\|\nabla v\|_{L^4} \non\\
&\leq& C(\|\nabla(\Delta d-f)\|+\|\Delta d-f\|)\|\nabla
v\|^{\frac12}\|\Delta v\|^{\frac12} \non\\
&\leq& \frac{\nu}{32}\|\Delta v\|^2+\frac{1}{32}\|\nabla(\Delta
d-f)\|^2+C\|\Delta d-f\|^4+C\|\nabla v\|^2.
  \eea
  A combination of all above estimates yields the inequality \eqref{he}. The proof is complete.
 \end{proof}
\br Since $\alpha \in [0, 1]$ is bounded, we see from the  proof in above that the
higher-order differential inequality \eqref{he} is uniform in $\alpha$.
\er

It is worth pointing out that on the
right-hand side of \eqref{he}, the term $A(t)$ replaces the constant
1 of the corresponding result in \cite{SL08} (for $\alpha=1$).
This improved inequality indeed has the advantage that it not only yields the uniform estimates of solutions to problem
\eqref{1}--\eqref{5}, which is important both in the proof for the global
existence result Theorem \ref{glo} and for the uniqueness Corollary \ref{uniq}, but also yields the decay property of velocity field $v$ (cf. \eqref{vcon1}) and helps to obtain the convergence rate (cf .\eqref{simplified higher-order inequality}).

 \bl \label{vcon}
 For any $t\geq 0$, the following uniform estimate holds
 \be
 \|v(t)\|_{H^1}+ \|d(t)\|_{H^2}\leq C, \label{ubdd}
 \ee
 where $C$ is a constant depending on $\nu, \|v_0\|_{H^1}, \|d_0
 \|_{H^2}$, but independent of $\alpha\in [0,1]$.  Furthermore,
 \be \lim_{t\rightarrow +\infty} (\|v(t)\|_{H^1}+ \|-\Delta d(t)+f(d(t))\|)=0. \label{vcon1}\ee
 \el
 \begin{proof}
 \eqref{intA} implies that $\int_0^{+\infty} A(t) dt<+\infty$. Then the uniform
 bound \eqref{ubdd} as well as the decay property \eqref{vcon1} follow from \eqref{he} and
 an analysis lemma \cite[Lemma 6.2.1]{Z04}.
 \end{proof}

\subsection{Convergence to Equilibrium}

The $\omega$-limit set of $(v_0,d_0)\in V\times
H^2_p(Q)\subset\subset L^2_p(Q) \times H^1_p(Q)$ is defined as
follows:
 \bea
 \omega(v_0,d_0) &= &\{ (v_\infty(x),d_\infty(x)) \mid\ \text{there
 \ exists\ } \{t_n\}\nearrow \infty  \text{\ such\ that\ } \non\\&&
 \ (v(t_n),d(t_n)) \rightarrow (v_\infty,d_\infty)\
 \text{in}\ L^2(Q) \times H^1(Q),\ \text{as}\ t_n\rightarrow +\infty
 \}.\non
 \eea
 We infer from Lemma  \ref{vcon} that
 \begin{proposition} \label{lim}  $\omega(v_0,d_0)$ is a nonempty bounded subset in $H^1_p(Q)\times
 H^2_p(Q)$, which is compact in $L_p^2(Q) \times H^1_p(Q)$. Besides, all asymptotic
 limiting points $(v_\infty, d_\infty)$ of problem \eqref{1}--\eqref{5} satisfy that $v_\infty=0$ and
 $d_\infty \in \mathcal{S}:=\{u\ | u \ \text{solves}\ \eqref{staaq} \}.$
 \end{proposition}
It remains to prove the convergence for the director field $d$. For
any initial data $(v_0,d_0)\in V\times H^2_p(Q)$, it follows from
Lemma \ref{vcon} that $ \|d\|_{H^2}$ is uniformly bounded.
 Proposition \ref{lim} implies that there is an increasing unbounded sequence
$\{t_n\}_{n\in\mathbb{N}}$ and a function $d_\infty\in \mathcal{S}$
such that
   \be \lim_{t_n\rightarrow +\infty} \|d(t_n)-d_\infty\|_{H^1}
   =0. \label{secon}
   \ee
Let us look at the following elliptic periodic boundary value
problem
\be \left\{\begin{array}{l}  - \Delta d + f(d)=0,\quad x\in \mathbb{R}^n, \\
   d(x+e_i)=d(x),\quad x\in \mathbb{R}^n.\\
   \end{array}
   \label{staaq}
  \right.
 \ee
Define
 \be
 E(d):=\frac12\|\nabla d\|^2 + \int_Q F(d)dx.\label{EDD}
 \ee
 One can easily verify that the solution to \eqref{staaq} is equivalent to a
 critical point of $E(d)$. Then we introduce a \L
 ojasiewicz--Simon type inequality that is related to our problem.
 \bl\label{ls}{\rm [\L ojasiewicz--Simon inequality]}
 Let $\psi$ be a critical point of $E(d)$. Then there exist constants
 $\theta\in(0,\frac12)$ and $\beta>0$ depending on $\psi$ such that
 for any $d\in H^1_p(Q)$ satisfying $\|d-\psi\|_{H_p^1(Q)}<\beta$, it
 holds
 \be
 \|-\Delta d+f(d)\|_{(H^1_p(Q))'}\geq
 |E(d)-E(\psi)|^{1-\theta},
 \ee
 where $(H^1_p(Q))'$ is the dual space of $H^1_p(Q)$.
 \el
 \br Lemma \ref{ls} can be viewed as an extended version of Simon's
result \cite{S83} for scalar functions under the $L^2$-norm. We
refer to \cite[Chapter 2, Theorem 5.2]{Hu} for a proof.
 \er

Since our system \eqref{1}-\eqref{5} enjoys exactly the same Lyapunov functional $\mathcal{E}(t)$ and the basic energy law \eqref{ENL} as the small molecule system studied in \cite{LL95, LW08}, based on this important observation, one can see that the application of \L
ojasiewicz--Simon inequality Lemma \ref{ls} indeed does not rely on
the special kinematic transport property of the flow.
Therefore, we can follow the steps in \cite[Section
3.1]{LW08} to prove that there exists a time $t_0>0$ such that for $t\geq t_0$, it holds
 \bea
 \frac{d}{dt}(\mathcal{E}(t)-E(d_\infty))^\theta+ C(\|\nabla v\|+\|-\Delta d+f(d)\|) \leq 0, \quad \forall\ t\geq
 t_0.\label{ly3}
 \eea
 Integrating with respect to time, we obtain
 \be  \int_{t_0}^{+\infty} (\|\nabla v(\tau)\|+ \|-\Delta
 d(\tau)+f(d(\tau))\|)d\tau<+\infty.\label{intvd}
 \ee
 Now we take the different kinematic transport of the director
 $d$ in equation \eqref{3} into account.  By the the Sobolev embedding theorem and uniform estimate \eqref{ubdd}, we conclude the following key observation
 \bea \| d_t\|
 &\leq& \|v\cdot \nabla d\|+\alpha\|(\nabla v)d\|+(1-\alpha)\|(\nabla^T v)d\|+ \|-\Delta
 d+f(d)\|\non\\
 & \leq & \|v\|_{L^4}\|\nabla d\|_{L^4}+ \|d\|_{L^\infty}\|\nabla v\|+ \|-\Delta d+f(d)\|\non\\
 &\leq & C\|\nabla v\|+ \|-\Delta d+f(d)\|,\label{dt}
 \eea
 which together with \eqref{intvd} yields that
 \be \int_{t_0}^{+\infty} \| d_t(\tau)\| d\tau<+\infty.
 \ee
 This fact and the sequence convergent result \eqref{secon} easily implies that
 \be \lim_{t\rightarrow +\infty}
 \|d(t)-d_\infty\|=0.
 \ee
 Since $d(t)$ is uniformly bounded in $H^2(Q)$ (cf. \eqref{ubdd}), by standard interpolation inequality we have
 \be \lim_{t\rightarrow +\infty}
 \|d(t)-d_\infty\|_{H^1}=0.\label{conh1}
 \ee
Moreover, the decay property of the quantity $A(t)$ (cf. Lemma
\ref{vcon}) will provide higher-order convergence of $d$. We notice that
 \bea \|\Delta d- \Delta d_\infty\|&\leq& \| \Delta d- \Delta d_\infty
 -f(d)+f(d_\infty)\|+ \|f(d)-f(d_\infty)\|\non
 \\
 &\leq& \|\Delta
 d-f(d)\|+\|f'(\xi)\|_{L^4}\|d-d_\infty\|_{L^4}\non\\
 &\leq& \|\Delta
 d-f(d)\|+C\|d-d_\infty\|_{H^1}.\label{kkk}
 \eea
 The above estimate together with \eqref{vcon1} and \eqref{conh1} yields
 \be \lim_{t\rightarrow +\infty}
 \|d(t)-d_\infty\|_{H^2}=0.\label{conh2}
 \ee

\subsection{Estimates on Convergence Rate}
We shall prove the estimate for convergence rate
\eqref{rate} using a similar argument as in
\cite{HJ01,WGZ1,W07,LW08} and references therein.
As has been shown in the literature, an estimate on the convergence rate in certain lower-order norm
could be obtained directly from the \L ojasiewicz--Simon approach
(cf. \cite{HJ01}). It
follows from the convergence of $(v,d)$ (cf. \eqref{cgce}) that the
decreasing energy functional $\mathcal{E}(t)$ satisfies
$$ \lim_{t\to+\infty} \mathcal{E}(t)= E(d_\infty), \quad \text{and}\quad \mathcal{E}(t)\geq E(d_\infty),\
 \forall t\geq 0.$$
  On the other hand, we infer from Lemma \ref{ls} and
\eqref{ly3} that
 \bea
 \frac{d}{dt}(\mathcal{E}(t)-E(d_\infty))+ C(\mathcal{E}(t)-E(d_\infty))^{2(1-\theta)}\leq 0, \quad \forall\ t\geq
 t_0.\label{ly4}
 \eea
As a consequence,
 $$ 0\leq \mathcal{E}(t)-E(d_\infty)\leq
 C(1+t)^{-\frac{1}{1-2\theta}},\quad \forall\ t\geq
 t_0. $$
  Integrating \eqref{ly3} on
$(t,+\infty)$, where $t\geq t_0$,  it follows from \eqref{dt} that
 \bea \int_t^{+\infty} \|d_t(\tau)\| d\tau & \leq& \int_t^{+\infty}(C\|\nabla v(\tau)\|+
 \|-\Delta d(\tau)+f(d(\tau))\|)d\tau\non\\
 &\leq&(\mathcal{E}(t)-E(d_\infty) )^\theta \leq
 C(1+t)^{-\frac{\theta}{1-2\theta}}.\non
  \eea
Adjusting the constant $C$ properly, we get
 \be
    \|d(t)-d_\infty\|\leq C(1+t)^{-\frac{\theta}{1-2\theta}}, \quad t\geq 0.\label{rate1}
 \ee

As in \cite{WGZ1,LW08}, higher-order estimates on the convergence rate
can be achieved by constructing proper differential inequalities via
energy method. Lemma \ref{vcon} implies that the steady state
corresponding to problem \eqref{1}--\eqref{5} can be
 reduced to the following system:
 \bea
 \nabla P_\infty+\nabla\left(\frac{|\nabla d_\infty|^2}{2}\right)&=&-\nabla d_\infty\cdot \Delta d_\infty,\label{s1a}\\
 -\Delta d_\infty+f(d_\infty)&=&0,\label{s2a}
 \eea
subject to periodic boundary conditions. In \eqref{s1a}, we have
used the fact that $\nabla\cdot(\nabla d_\infty\odot\nabla
d_\infty)=\nabla\left(\frac{|\nabla d_\infty|^2}{2}\right) + \nabla
d_\infty\cdot\Delta d_\infty$ (cf. \cite{LL95}). Subtracting the
stationary problem \eqref{s1a}--\eqref{s2a} from the evolution
problem \eqref{1}--\eqref{3}, we obtain that
 \bea
 && v_t+v\cdot\nabla v-\nu \Delta v+\nabla (P-P_\infty)
 + \nabla \left(\left(\frac{|\nabla d|^2}{2}\right)-\left(\frac{|\nabla
 d_\infty|^2}{2}\right)\right)\non\\
  &=&-\nabla\cdot[\alpha (\Delta d-f(d))\otimes
 d-(1-\alpha)d\otimes (\Delta d-f(d))]-\nabla d\cdot \Delta d +\nabla d_\infty\cdot \Delta d_\infty
 ,\label{11}
 \eea
 \vspace{-12mm}
 \bea
  \qquad \ \ \nabla \cdot v &=& 0,\label{22}\\
  d_t+v\cdot\nabla d-\alpha (\nabla v)d+(1-\alpha)(\nabla^T v)d&=&\Delta (d-d_\infty)
 -f(d)+f(d_\infty).\label{33}
  \eea
Multiplying \eqref{11} by $v$ and \eqref{33} by $-\Delta
d+f(d)=-\Delta (d-d_\infty)+f(d)-f(d_\infty)$, respectively,
integrating over $Q$, and adding the results together, we have
 \bea
  && \frac{d}{dt} \left(\frac12\|v\|^2+\frac12\|\nabla d-\nabla
  d_\infty\|^2+\int_Q [F(d)-F(d_\infty)- f(d_\infty)(d-d_\infty)]dx\right)\non\\
  && \quad +\nu\|\nabla v\|^2+ \|\Delta d-f(d)\|^2\non\\
  &=& (v, \nabla d_\infty\cdot \Delta d_\infty)=
   (v, \nabla d_\infty\cdot (\Delta d_\infty-f(d_\infty)))+ (v, \nabla F (d_\infty)) \non\\
  &=&0.
  \label{ra1}
  \eea
 Multiplying \eqref{33} by $d-d_\infty$ and integrating over $Q$,
we obtain
 \bea && \frac12\frac{d}{dt}\|d-d_\infty\|^2+\|\nabla (d-d_\infty)\|^2\non\\
 &=&
 -(v\cdot\nabla d,d-d_\infty) + \alpha (( \nabla v)d, d-d_\infty)-(1-\alpha)(( \nabla^T v)d, d-d_\infty)\non\\
 &&  -(f(d)-f(d_\infty),
 d-d_\infty)\non\\
 &:=&J_1.\label{ra3}
 \eea
 It follows from the uniform estimates
 \eqref{ubdd} and $\alpha \in [0,1]$ that
  \bea
 |J_1|&\leq& \|v\|_{L^4}\|\nabla
 d\|_{L^4}\|d-d_\infty\|+\|\nabla v\|\|d\|_{L^\infty}\|d-d_\infty\| +\|f'(\xi)\|_{L^\infty}\|d-d_\infty\|^2\non\\
 &\leq& C\|\nabla v\|\|d-d_\infty\|+
 C\|d-d_\infty\|^2\leq \varepsilon_1\|\nabla v\|^2+C\|d-d_\infty\|^2,\label{ra4}
 \eea
 where $\xi=ad+(1-a)d_\infty$, $a\in (0,1)$.\\
 Multiplying \eqref{ra3} by $\mu>0$ and adding the resultant to
\eqref{ra1}, using \eqref{ra4}, we get
 \bea
  && \frac{d}{dt} \left(\frac12\|v\|^2+\frac12\|\nabla d-\nabla
  d_\infty\|^2+ \frac{\mu}{2} \|d-d_\infty\|^2 +\int_Q (F(d)-F(d_\infty)) dx\right.\non\\
  &&\ \ \left.-
  \int_Q f(d_\infty)(d-d_\infty)dx\right)+\left(\nu-\mu\varepsilon_1\right)\|\nabla v\|^2
  + \|\Delta d-f(d)\|^2  +\mu\|\nabla(d-d_\infty)\|^2  \non\\
  &\leq& C_1\mu \|d-d_\infty\|^2.\label{ra5}
  \eea
By the Taylor's expansion, we deduce that
 \be \left\vert\int_Q
 [F(d)-F(d_\infty)-f(d_\infty)(d-d_\infty)]
         dx\right\vert\leq
          \|f'(\xi)\|_{L^\infty}\|d-d_\infty\|^2 \leq
          C_2\|d-d_\infty\|^2,
          \label{rate6}
  \ee
  where $\xi=ad+(1-a)d_\infty$, $a\in (0,1)$.
Let us define now, for $t\geq 0$,
 \bea y(t)&= & \frac12\|v(t)\|^2+\frac12\|\nabla d(t)-\nabla
  d_\infty\|^2+ \frac{\mu}{2} \|d(t)-d_\infty\|^2 +\int_Q (F(d(t))dx- F(d_\infty)) dx\non\\
  &&\ \ -
  \int_Q f(d_\infty)(d(t)-d_\infty)dx. \label{y1}
  \eea
In \eqref{y1}, if we choose $ \mu \geq  1+2C_2>0$, then there exists a constant $k>0$ that
  \be k(\|v(t)\|^2+\|d(t)-d_\infty\|_{H^1}^2)\geq  y(t)\geq \frac12(\|v(t)\|^2+\|d(t)-d_\infty\|_{H^1}^2).
 \label{y1a}
  \ee
 Next, in \eqref{ra5} we take $\varepsilon_1=\frac{\nu}{2\mu}$, then we can infer from \eqref{ra5} and \eqref{y1a} that there exist constants $C_3,C_4,C_5>0$ such that the following inequality holds
  \be
  \frac{d}{dt} y(t)+C_3 y(t)+C_4A(t)\leq C_5\|d(t)-d_\infty\|^2\leq
  C(1+t)^{-\frac{2\theta}{1-2\theta}}.\label{ra7}
  \ee
  It follows that (cf. \cite{W07,WGZ1,LW08})
  \be y(t)\leq C(1+t)^{-\frac{2\theta}{1-2\theta}}, \quad \forall\ t\geq 0,\label{ratey}
  \ee
  which together with \eqref{y1a} implies that
  \be \|v(t)\|+\|d(t)-d_\infty\|_{H^1}\leq C(1+t)^{-\frac{\theta}{1-2\theta}}, \quad \forall\ t\geq 0.\label{rate2}
  \ee

Finally, we prove the convergence rate of $(v, d)$ in $V\times H^2$.
It follows from Lemma \ref{he2d} and Lemma \ref{vcon} that
 \be \frac{d}{dt}A(t) \leq C_6A(t).
  \label{simplified higher-order inequality}
  \ee
Multiplying \eqref{simplified higher-order inequality} with
$\alpha_1=\frac{C_4}{2C_6}$, adding with \eqref{ra7}, we arrive at
 \be \frac{d}{dt}[y(t)+\alpha_1A(t) ]+C[y(t)+\alpha_1A(t) ] \leq
C(1+t)^{-\frac{2\theta}{1-2\theta}}, \ee
 from which we conclude that
 \be  y(t)+\alpha_1 A(t) \leq
C(1+t)^{-\frac{2\theta}{1-2\theta}}, \ \ \ \ \ \forall t \geq 0.
 \ee
This together with \eqref{y1a} yields
   \be \|\nabla v(t)\|^2+\|\Delta d(t)-f(d(t))\|^2=A(t)\leq C(1+t)^{-\frac{2\theta}{1-2\theta}}, \quad \forall\ t\geq 0.\label{rate3a}
  \ee
Recalling \eqref{kkk}, it follows from \eqref{rate3a} and \eqref{rate2} that
 \be \|\Delta d(t)-\Delta d_\infty\| \leq C(1+t)^{-\frac{\theta}{1-2\theta}}, \quad \forall\ t\geq 0.\label{rate4a}
  \ee
Summing up, we can deduce the required estimate \eqref{rate} from
\eqref{rate2}, \eqref{rate3a} and \eqref{rate4a}.

The proof of Theorem \ref{main2d} is complete. \ \ $\square$

\section{Results in Three Dimensional Case }
\setcounter{equation}{0}

In this section, we study the problem in 3D. First, the following
higher-order energy inequality will enable us to conclude the
existence and uniqueness of a local strong solution to problem
\eqref{1}-\eqref{3} in $3D$.

 \bl\label{highor3d}
In $3D$ case, the following inequality holds for the
 classical solution $(v,d)$ to problem \eqref{1}--\eqref{5}:
\begin{eqnarray}
\frac{d}{dt}A(t)+ \nu\|\Delta v(t)\|^2+ \|\nabla (\Delta
d-f(d(t)))\|^2 \leq C_{\ast}(A(t)^4+A(t)), \label{Eins}
\end{eqnarray}
where $A(t)$ is defined as in \eqref{A} and $C_{\ast}>0$ is a
constant that only depends on $\nu$, $f$, $Q$, $\|v_0\|$ and
$\|d_0\|_{H^1}$ but not on $\alpha$.
 \el
\begin{proof}
First, like in $2D$, we still have the uniform $L^2\times H^1$
estimate for the solution $(v, d)$ to problem \eqref{1}--\eqref{5}
(cf. Lemma \ref{v0h1})
 $$
\|v(t)\|+\|d(t)\|_{H^1}\leq C,\quad \forall\ t\geq 0,
 $$
 where $C$ is a constant only depending on
$\|v_0\|$ and $\|d_0\|_{H^1}$. Recalling \eqref{diffA} and keeping
\eqref{cancellation 1}, \eqref{cancellation 2} in mind, we have
 \be
 \frac12\frac{d}{dt}A(t)+(\nu\|\Delta v\|^2+\|\nabla (\Delta-f(d))\|^2)=I_1+I_3'+I_4'+I_{5b}+I_{5c}'+I_6'+I_7'+I_8'.
 \label{3r1a}
 \ee
 We re-estimate the right-hand side of \eqref{3r1a} term by term.
 \bea |I_1| &\leq& \|\Delta v\|\|v\|_{L^\infty}\|\nabla
v\|\leq C\|\Delta
v\|(\|\Delta v\|^{\frac{3}{4}}+1)\|\nabla v\|\non\\
& \leq &\frac{\nu}{16}\|\Delta v\|^2+\frac{C}{\nu^7}\|\nabla
v\|^8.\non
 \eea
 In $3D$ case, we have
 \bea
 \|\nabla \Delta d\|&\leq&  \|\nabla (\Delta d-f(d))\|+\|f'(d)\|_{L^3}\|\nabla
 d\|_{L^6}\leq \|\nabla (\Delta d-f(d))\|+C(\|\Delta d\|+1)\non\\
 &\leq& \|\nabla (\Delta d-f(d))\|+ C\|\Delta d-f(d)\|+ C.\label{ndd3d}
 \eea
 As a result,
 \bea
 && |I_3'|+|I_4'|\non\\
 &\leq &\alpha|(\Delta d-f(d), \Delta d\cdot \nabla v)|+(1-\alpha)|(\Delta d-f(d), \Delta d\cdot \nabla^T v)|\non\\
 &\leq& \|\Delta d-f(d)\|_{L^3}\|\Delta d\|\|\nabla v\|_{L^6} \non\\
 &\leq& C(\|\nabla (\Delta d-f(d))\|^\frac12\|\Delta
 d-f(d)\|^\frac12+ \|\Delta d-f(d)\|)(\|\Delta d-f(d)\|+1)\|\Delta
 v\|\non\\
  &\leq & \frac{1}{10}\|\nabla (\Delta
 d-f(d))\|^2+\frac{\nu}{16}\|\Delta v\|^2+C(\|\Delta
 d-f(d)\|^6+\|\Delta
 d-f(d)\|^2).\label{ii34}
 \eea
 \bea
  && |I_6'|+|I_7'|\non\\
 &\leq & 2\alpha|(\Delta d-f(d), (\nabla
 d\cdot\nabla )\cdot\nabla v)|+ 2(1-\alpha)|(\Delta d-f(d), (\nabla
 d\cdot\nabla )\cdot\nabla^T v)|\non\\
  &\leq& C\|\Delta
d-f\|_{L^3}\|\nabla d\|_{L^6}\| v\|_{H^2} \non\\
&\leq& C(\|\nabla (\Delta d-f(d))\|^\frac12\|\Delta
 d-f(d)\|^\frac12+ \|\Delta d-f(d)\|)(\|\Delta d-f(d)\|+1)\|\Delta v\| \non\\
&\leq& \frac{1}{10}\|\nabla (\Delta
 d-f(d))\|^2+\frac{\nu}{16}\|\Delta v\|^2+C(\|\Delta
 d-f(d)\|^6+\|\Delta
 d-f(d)\|^2).\label{ii67}
 \eea
 Next,
 \bea |I_{5b}| &\leq&
2\|\Delta d-f\|_{L^3}\|\nabla v\|_{L^6}\|\nabla^2 d\|  \non\\
&\leq&  C(\|\nabla(\Delta d-f)\|^\frac12\|\Delta
d-f\|^\frac12+\|\Delta d-f\|) \|\Delta v\|(\|\Delta
d-f\|+1)  \non\\
&\leq&\frac{1}{10}\|\nabla (\Delta
 d-f(d))\|^2+\frac{\nu}{16}\|\Delta v\|^2+C(\|\Delta
 d-f(d)\|^6+\|\Delta
 d-f(d)\|^2).\label{ii5b}
 \eea
\bea
 |I_{5c}'| &\leq&
 \|f'(d)\|_{L^\infty}\|\Delta d-f\|^2\leq C(\|d\|_{L^\infty}^2+1)\|\Delta d-f\|^2 \non\\
&\leq& C(\|d\|_{H^2}+1) \|\Delta d-f\|^2\leq C(\|\Delta d-f\|+1)
\|\Delta d-f\|^2
 \eea
 \bea |I_8'| &\leq& \|\Delta
d-f\|_{L^6}\|f'(d)\|_{L^3}\|d\|_{L^\infty}\|\nabla v\| \non\\
&\leq& C(\|\nabla(\Delta d-f)\| +\|\Delta d-f\|) (\|\Delta
d-f\|+1)\|\nabla
v\| \non\\
&\leq& \frac{1}{10}\|\nabla(\Delta d-f)\|^2+C(\|\Delta
d-f\|^4+\|\Delta d-f\|^2+\|\nabla v\|^4+\|\nabla v\|^2).\label{ii8}
  \eea
Collecting all the estimates above, we can conclude \eqref{Eins}.
The proof is complete.
  \end{proof}
Next, we show the following result that is useful in understanding
the long-time behavior of system \eqref{1}--\eqref{5} (we refer to
\cite{LL95} for the corresponding result of the small molecule
system):
  \begin{theorem}\label{sm}
For any $R>0$ and initial data $(v_0, d_0)\in V \times H_p^2(Q)$,
whenever $\|\nabla v_0\|^2+\|\Delta d_0-f(d_0)\|^2\leq R$, there
exists a constant $\varepsilon_0 \in(0,1)$, depending on $\nu$, $f$,
$Q$ and
$R$, such that either\\
(a) for the energy functional
$\mathcal{E}(t)=\frac12\|v(t)\|^2+\frac12\|\nabla
d(t)\|^2+\int_{Q}F(d(t))\,dx$, there exits $T_*\in (0,+\infty)$ such
that
\[\mathcal{E}(T_*) \leq \mathcal{E}(0)-\varepsilon_0,  \]
or\\
(b) problem \eqref{1}--\eqref{5} has a unique global classical
solution $(v,d)$ with uniform estimate
 \begin{equation}
 \|v(t)\|_{H^1(Q)}+\|d(t)\|_{H^2(Q)}\leq C,\quad \forall\ t\geq0.
 \end{equation}
\end{theorem}
\begin{proof}
\textbf{Step 1}. A local well-posedness result. We consider the
following initial value problem of a nonlinear ODE:
$$ \frac{d}{dt}Y(t)=C_*(Y(t)^4+Y(t)),\quad Y(0)=A(0)=\|\nabla v_0\|^2+\|\Delta d_0-f(d_0)\|^2 \leq R.$$
 We denote by $I=[0,T_{max})$ the maximum existence interval of
$Y(t)$ such that
$$
\lim_{t\rightarrow T_{max}^-} Y(t)=+\infty.
$$
 It easily follows from \eqref{Eins} that for any $t\in I$, $0\leq
A(t)\leq Y(t)$.
 Consequently, $A(t)$ exists on $I$. Moreover,
 $T_{max}$ is determined by $Y(0)$ and $C_*$ such that $T_{max}=T_{max}(Y(0),C_*)$ is
 increasing when $Y(0)\geq 0$ is decreasing. Taking $t_0=\frac12 T_{max}(R, C_*)> 0$, then it
 follows that $Y(t)$ as well as $A(t)$ is uniformly bounded on
$[0, t_0]$. This fact together with the approximation procedure in
\cite{SL08} and Lemma \ref{cd} implies the local existence of a
unique (classical) solution of problem \eqref{1}--\eqref{5} at least
on $[0,t_0]$.

\textbf{Step 2}. If $(a)$ is not true, we have for all $t\geq 0$, $
\mathcal{E}(t) \geq \mathcal{E}(0)-\varepsilon_0$. From the basic
energy law \eqref{ENL}, we infer that
\[\int_0^{+\infty}(\nu\|\nabla v(t)\|^2+\|\Delta d(t)-f(d(t))\|^2)dt
\leq \varepsilon_0.  \]
 Hence, there exists
$t_* \in [\frac{t_0}{2}, t_0]$ that
\[ \nu\|\nabla v(t_*)\|^2+\|\Delta
d(t_*)-f(d(t_*))\|^2 \leq \frac{2\varepsilon_0}{t_0}.  \]
 Choosing
$\varepsilon_0>0$ such that
$$\frac{2}{\min\{1,\nu\}}\frac{\varepsilon_0}{t_0} \leq R,$$
we have $A(t_*)\leq R$. Taking $t_*$ as the initial time, we infer
from the above argument that $A(t)$ is uniformly bounded at least on
$[0, \frac{3t_0}{2}]\subset [0,t_*+t_0]$. Moreover, its bound only
depends on $R$ and $C_*$, but not on the length of existence
interval. Therefore, we can extend the local solution obtained in
step 1 to infinity such that
 \be A(t)\leq
C,\quad \forall\ t\geq0, \label{bdA}
 \ee
where $C$ is uniform in time. The proof is complete.
\end{proof}

Theorem \ref{sm} implies that if the energy $\mathcal{E}$ does not
"drop" too fast, problem \eqref{1}--\eqref{5} admits a unique global
classical solution. This assumption can be verified for certain
special cases, which are stated in the following corollaries.

\begin{corollary}\label{near}
 Let $d^*\in H_p^2(Q)$ be an absolute minimizer of the functional
 $ E(d)=\frac12\|\nabla d\|^2+ \int_Q F(d)dx$. There exists a constant $\sigma\in (0,1]$ that may depend on $\nu$, $f$, $Q$ and $d^*$
 such that for any initial data $(v_0, d_0)\in V\times H^2_p(Q)$
 satisfying $\|v_0\|_{H^1}+\|d_0-d^*\|_{H^2}\leq\sigma$,  problem
 \eqref{1}--\eqref{5} admits a unique global classical solution.
 \end{corollary}
\begin{proof}
Without loss of generality, we assume that $\sigma \leq 1$. From the
assumption $\|v_0\|_{H^1}+\|d_0-d^\ast\|_{H^2} \leq \sigma\leq 1$,
we infer that
\begin{eqnarray}
 \|v_0\|^2_{H^1}+\|\Delta d_0-f(d_0)\|^2
 &\leq&
\|v_0\|^2_{H^1}+2\|\Delta d_0-\Delta d^*\|^2+2\|f(d_0)-f(d^*)\|^2 \non\\
&\leq& K_1(\|v_0\|_{H^1}+\|d_0-d^*\|_{H^2})^2 \leq  K_1.
\end{eqnarray}
 In addition, since $d^*$ is the absolute minimizer of $E(d)$,
we have
\begin{eqnarray*}
 \mathcal{E}(0)-\mathcal{E}(t)&\leq& \mathcal{E}(0)- E(d(t))\leq
 \mathcal{E}(0)-E(d^*)\nonumber
\\ &\leq& \frac12\|v_0\|^2+ \frac12\left(\|\nabla d_0\|^2-\|\nabla d^*\|^2\right)+ \int_Q
F(d_0)-F(d^*) dx\nonumber\\
& \leq & \frac12 \sigma^2+ C\sigma\leq   K_2\sigma.
\end{eqnarray*}
 Here $K_1$ and $K_2$ are positive constants that only depend on
 $d^*$, $\nu$, $f$ (not on $\sigma$).

Now we take
\begin{equation}
 R=K_1,\quad \varepsilon_0=K_2\sigma,\quad \sigma=\min\left\{1,\ \frac{K_1}{4K_2}T_{max}(K_1,C_*)\min\{1,\nu\}\right\},\label{small}
 \end{equation}
 then the conclusion follows from  Theorem \ref{sm}.
\end{proof}

Corollary \ref{near} implies that if the initial velocity $v_0$ is
small in $H^1$ and the initial director field $d_0$ is properly
close to the absolute minimizer $d^*$ of the functional $E(d)$ in
$H^2$, problem \eqref{1}--\eqref{5} admits a unique global classical
solution. However, from the proof of Theorem \ref{sm} we can
somewhat relax the "smallness" requirement on $(v_0,d_0)$ from
$H^1\times H^2$ to $L^2\times H^1$.

\begin{corollary}\label{near1}
 Let $d^*\in H_p^2(Q)$ be an absolute minimizer of the functional
 $ E(d)$.
 For any initial data $(v_0, d_0)\in V\times H_p^2(Q)$, there exists
 a constant $\sigma\in (0,1]$, which depends on $\nu$, $f$, $Q$,
 $d^*$, $\|v_0\|_{H^1}$ and $\|d_0\|_{H^2}$ such that if
 $\|v_0\|+\|d_0-d^*\|_{H^1}\leq\sigma$,
  problem
 \eqref{1}--\eqref{5} admits a unique global classical solution.
 \end{corollary}
 \begin{proof}
 Without loss of generality, we assume that $\sigma\leq1$. Set
 $K_1:=\|\nabla v_0\|^2+\|\Delta d_0-f(d_0)\|^2<\infty$ (unlike in Corollary
 \ref{near}, now $K_1$ depends on $\|v_0\|_{H^1}$ and
 $\|d_0\|_{H^2}$). And we have $\mathcal{E}(0)-\mathcal{E}(t)\leq K_2\sigma$,
where $K_2$ is a positive constant only depending on
 $d^*$, $\nu$, $f$ (not on $\sigma$). As in the proof of Corollary \ref{near}, we take $R=K_1$,
$\varepsilon_0=K_2\sigma$ and $
 \sigma=\min\left\{1,\
 \frac{K_1}{4K_2}T_{max}(K_1,C_*)\min\{1,\nu\}\right\}$.
 The conclusion follows from Theorem \ref{sm}. Here, we note that now
 $\sigma$ depends on $\|v_0\|_{H^1}$ and $\|d_0\|_{H^2}$ (because of
 $K_1$) while in Corollary \ref{near}, $\sigma$ does not.
 \end{proof}

By so far, we discussed global existence results when the initial
data are assumed to be close to certain equilibria. On the other
hand, concerning arbitrary initial data in $3D$, the global
well-posedness can still be obtained if we assume that the viscosity
is properly large (cf. \cite{LL95} for small molecule system and
\cite{SL08} for the case $\alpha=1$). The (constant) viscosity $\nu$ plays an
essential role in the proof that the largeness of $\nu$ will
guarantee the existence of global strong solutions. To see this
point, we show that the following higher-order differential
inequality holds uniformly in $\alpha\in [0,1]$.

\bl \label{he3da} In $3D$ case, for arbitrary $\nu_0>0$, if $\nu\geq
\nu_0>0$, then the following inequality holds for the
 classical solution $(v,d)$ to problem \eqref{1}--\eqref{5}:
 \be \frac{d}{dt}\tilde{A}(t) \leq
 -\left(\nu-M_1\nu^\frac12\tilde{A}(t)\right)\|\Delta v\|^2-\left(1-M_1 \nu^{-\frac14} \tilde{A}(t)\right)
 \|\nabla(\Delta d-f(d))\|^2+M_2A(t),\label{he3d}
 \ee
 where $\tilde{A}(t)=A(t)+1$ (cf. \eqref{A} for the definition of $A(t)$), $M_1,M_2$ are positive
  constants depending on $f, |Q|, \|v_0\|, \|d_0
 \|_{H^1}$, $M_2$  also depends on $\nu_0$, the lower bound of $\nu$.
 \el
 \begin{proof} We re-estimate the terms on the right-hand-side of \eqref{3r1a} in a different way.
 Notice that we still have the uniform $L^2\times H^1$ estimate for the solution $(v, d)$ as
 before. Then we have
\bea |I_1|
 &\leq& \|v\|_{L^4}\|\nabla v\|_{L^4}\|\Delta v\| \leq C\|v\|^{\frac{1}{4}}\|\nabla v\|^{\frac{3}{4}}\|\nabla
v\|^{\frac{1}{4}}\|\Delta v\|^{\frac{3}{4}}\|\Delta v\| \non\\
&\leq& {\nu}^{\frac{1}{2}}\|\Delta v\|^2+{\nu}^{\frac{1}{2}}\|\nabla
v\|^2\|\Delta v\|^2+C{\nu}^{-\frac{7}{2}}\|\nabla v\|^2.\non
 \eea
 It follows from \eqref{ii34}, \eqref{ii67}, \eqref{ii5b} and \eqref{ii8} that all the five terms $|I_3'|, |I_4'|, |I_{5b}|, |I_6'|, |I_7'|$ can be bounded by $C(\|\Delta d-f\|^{\frac12}\|\nabla(\Delta
d-f)\|^{\frac12}+\|\Delta
d-f\|)(\|\Delta d -f\|+1)\|\Delta v\|$. Therefore, 
\bea
&&|I_3'|+ |I_4'|+ |I_{5b}|+ |I_6'|+ |I_7'| 
\non\\
&\leq&
C(\|\Delta d-f\|^{\frac12}\|\nabla(\Delta
d-f)\|^{\frac12}+\|\Delta
d-f\|)(\|\Delta d -f\|+1)\|\Delta v\|\non\\
 &\leq&
 \nu^\frac12 (1+ \|\Delta d-f\|^2)\|\Delta v\|^2+  \nu^{-\frac14}(1+\|\Delta d
 -f\|^2) \|\nabla(\Delta d-f)\|^2\non\\
 && +C\left(1+\frac{1}{\nu^{\frac12}}+\frac{1}{\nu^{\frac34}}\right)\|\Delta d -f\|^2.\non
 \eea
Next,
\bea
 |I_{5c}'| &\leq&  \|f'(d)\|_{L^3}\|\Delta d-f\|_{L^6}\|\Delta
 d-f\|\leq  C(\|\nabla(\Delta d-f)\|+\|\Delta d-f\|)\|\Delta d-f\| \non\\
&\leq& \frac{1}{16} \|\nabla(\Delta d-f)\|^2+C\|\Delta d-f\|^2.\non
 \eea
 \bea |I_8'| &\leq& \|\Delta
d-f\|_{L^6}\|f'(d)\|_{L^3}\|d\|_{L^\infty}\|\nabla v\| \non\\
&\leq& C(\|\nabla(\Delta d-f)\| +\|\Delta d-f\|) (\|\Delta
d-f\|+1)\|\Delta
v\| \non\\
&\leq& \nu^\frac12 (1+ \|\Delta d-f\|^2)\|\Delta v\|^2+  \nu^{-\frac14}(1+\|\Delta d
 -f\|^2) \|\nabla(\Delta d-f)\|^2\non\\
 && +C\left(1+\frac{1}{\nu^{\frac12}}+\frac{1}{\nu^{\frac34}}\right)\|\Delta d -f\|^2.\non
  \eea
Putting all the above estimates together, we arrive at \eqref{he3d}.
\end{proof}

In what follows, we proceed to prove Theorem 1.2 and Theorem 1.3.

\begin{lemma}
\label{AA1} In $3D$ case, for the unique strong solution $(v,d)$
obtained in Theorem \ref{glo}, Corollary \ref{near} and
Corollary \ref{near1}, it holds
 \be \lim_{t\rightarrow+\infty} A(t)=0. \label{AAc}
 \ee
\end{lemma}
\begin{proof}
 (1) \textbf{Near equilibrium case}. When the assumptions
 in Theorem \ref{sm} (or Corollary \ref{near} / Corollary \ref{near1}) are satisfied, we
 have \eqref{bdA} holds for all $t\geq0$. Thus, it follows from
 \eqref{Eins} that for all $t\geq 0$,
 $$ \frac{d}{dt}A(t)\leq C_*(A^4(t)+A(t)) \leq
 C.$$ Recalling the fact that  $A(t)\in L^1(0,+\infty)$
(cf. \eqref{intA}), we conclude \eqref{AAc}.

 (2) \textbf{Large viscosity case}. We infer from \eqref{intA} that $$\int_t^{t+1}\tilde{A}(t)
dt=\int_t^{t+1}A(t) dt+1 \leq \int_0^{+\infty}A(t) dt
+1:=\tilde{M}.$$ Based on Lemma \ref{he3da}, if we assume that the
viscosity $\nu$ satisfies the following relation
 \be \nu^\frac14\geq \nu_0^\frac14:=
M_1(\tilde{A}(0)+M_2\tilde{M}+4\tilde{M})+1,\label{nu}
 \ee
by applying the same idea as in \cite{LL95,LW08}, we can show that
 $ \tilde{A}(t)$ uniformly bounded for all $t\geq 0$. Moreover, in this case, for $t\geq 0$, we have $\left(\nu-M_1\nu^\frac12\tilde{A}(t)\right)\geq 0$ and
  $\left(1-\frac{M_1}{\nu^\frac14}\tilde{A}(t)\right)\geq 0$ and it follows from \eqref{he3d} that
\[ \frac{d}{dt}A(t)=\frac{d}{dt}\tilde{A}(t) \leq
M_2A(t) \leq C.\] By the same argument as in (1), we obtain
\eqref{AAc}.  The proof is complete.
\end{proof}

\textbf{Proof of Theorem 1.2 and Theorem 1.3.} Based on Lemma
\ref{AA1}, for both (1) near equilibrium case and (2) large
viscosity case, one can argue exactly as in Section 3.2 to conclude
that
 \be \lim_{t\rightarrow +\infty}
 (\|v(t)\|_{H^1}+\|d(t)-d_\infty\|_{H^2})=0.\label{cgce3d}
 \ee
Since we have obtained uniform bounds for $\|v(t)\|_{V}$ and
$\|d(t)\|_{H^2}$, we are able to show the estimate on convergence
rate \eqref{rate} for both cases. In order to see this, one can
check the argument for $2D$ case step by step. By applying
corresponding Sobolev embedding theorems in $3D$, we can see that
all calculations in Section 3.3 are valid (the details are omitted
here). We complete the proof for Theorem \ref{main3da} and Theorem
\ref{main3d}. $\square$
\bigskip

\section{Remark on Liquid Crystal Flows with Non-vanishing Average Velocity}
\setcounter{equation}{0}

 We briefly
discuss the flows with non-vanishing average velocity. Due to the
periodic boundary condition \eqref{4}, by integration of
 \eqref{1} over $Q$, we get
 \be \frac{d}{dt}\left(\frac{1}{|Q|}\int_Q v(t) dx\right)=0,\ee
 which implies
 \be m_v:=\frac{1}{|Q|}\int_Q v(t) dx\equiv\frac{1}{|Q|}\int_Q
 v_0
 dx,\quad \forall\ t\geq 0,\label{mv}
 \ee
 where $|Q|$ is the measure of $Q$.

 We note that our main results (Theorems 1.1--1.3) are valid for the flow with vanishing average velocity
 (see the definition of the function space $V$), namely, $m_v=0$.
 In that case, we can apply the Poincar\'{e} inequality to
 $v\in V$ such that $\|v\|\leq C\|\nabla v\|$. This enables us to
 show the convergence property of the fluid velocity as well as the director field
  as $t\to +\infty$, due to the dissipative mechanism of system
 \eqref{1}--\eqref{5}.

 When a flow with non-vanishing average velocity $v$ is considered, as for the single Navier--Stokes equation (cf. \cite{Te}),
  we set
  \be
  v=\tilde{v}+m_v.\label{tran}
  \ee
  Then we transform problem \eqref{1}--\eqref{5} into the following system for new variables
 ($\tilde{v}, d$):
 \bea
 && \tilde{v}_t+\tilde{v}\cdot\nabla \tilde{v}-\nu \Delta \tilde{v}+m_v\cdot \nabla \tilde{v}+\nabla P\non\\
 &&\ =-\lambda
 \nabla\cdot[\nabla d\odot\nabla d+\alpha(\Delta d-f(d))\otimes d-(1-\alpha)d\otimes(\Delta d-f(d))\ \!],\label{1b}\ \!\\
 && \nabla \cdot \tilde{v} = 0,\label{2b}\\
 && d_t+\tilde{v}\cdot\nabla d+ m_v\cdot \nabla d-\alpha (\nabla \tilde{v})d +(1-\alpha)(\nabla^T \tilde{v})d
  =\gamma(\Delta d-f(d)),\label{3b}
 \eea
 subject to the corresponding periodic boundary conditions and initial conditions
 \bea
 &&\tilde{v}(x+e_i)=\tilde{v}(x),\quad d(x+e_i)=d(x),\qquad \text{for}\ x\in \mathbb{R}^n,
 \label{4b}\\
 &&\tilde{v}|_{t=0}=\tilde{v}_0(x)=v_0(x)-m_v, \ \ \text{with}\
 \nabla\cdot \tilde{v}_0=0,\ \
 d|_{t=0}=d_0(x),\ \  \text{for}\ x\in Q.\label{5b}
 \eea
 It is not difficult to check that system \eqref{1b}--\eqref{5b}
 still
 enjoys the  \textit{basic energy law}
 $$
 \frac{d}{dt}\tilde{\mathcal{E}}(t)=-\nu\|\nabla \tilde{v}\|^2-\lambda\gamma\|\Delta
 d-f(d)\|^2,
 \quad \forall \ t\geq 0,
 $$
 where
 $$ \tilde{\mathcal{E}}(t)=\frac{1}{2}\|\tilde{v}\|^2+\frac{\lambda}{2}\|\nabla
 d\|^2+\lambda\int_Q F(d)dx.
 $$
 By a similar argument, we can still prove the global existence and uniqueness of classical
 solutions $(\tilde{v}, d)$ to problem \eqref{1b}--\eqref{5b} under the same
 assumptions as in Theorem \ref{glo}, Theorem \ref{sm}, Corollary \ref{near} and Corollary \ref{near1}.
 Moreover, we can prove the same higher-order energy inequalities like
 Lemma \ref{he2d}, Lemma \ref{highor3d} and Lemma \ref{he3da} for $(\tilde{v}, d)$.

 As far as the
 long-time behavior of the global solution is concerned, following
 a similar argument in the previous sections, we can conclude that
 \be \lim_{t\rightarrow+\infty}(\|\tilde{v}(t)\|_{H^1}+\|\Delta
 d(t)-f(d(t))\|)=0.\label{imvd}
 \ee
 Recalling \eqref{tran}, we infer from \eqref{imvd} that $$
 \lim_{t\rightarrow+\infty}\|v(t)-m_v\|_{H^1}=0.$$
 However, in general we are not able to conclude similar results on the convergence of $d$ like in Theorems 1.1--1.3.
 \eqref{imvd} implies that the 'limit' function of $d(t)$ as $t\to +\infty$, which is denoted by $\hat d$, will
 satisfy $\Delta \hat d-f(\hat d)=0$ with corresponding periodic boundary
 condition. Let us look at the 'limiting' case such that $v=\hat v = m_v$
 and $d=\hat d$. It follows from \eqref{5} that
 \be \frac{D}{Dt}\hat d=\hat d_t+\hat v\cdot \nabla \hat d=0.
 \ee
 Consequently, $\hat d$ is purely transported and it $(i.e., \hat d(x(X,t),t))$ remains
 unchanged when the molecule moves through a flow
 field with velocity $m_v$. However, the local rate of change $\hat d_t$ may not be zero, because the convective rate of
 change may not vanish. Hence, in the Eulerian coordinates, or in $Q$, $\hat d(x, t)$ may change in time. As a result,
 there might be no steady state for the director field. Obviously, this is different from the
 situation in the previous sections, where all the three rates of change are vanishing in the limiting
 case. We can look at a simple example. In the case of periodic boundary conditions, let $\hat
 v=(1,0)$ and $\hat d(x,0)=\hat d_0(x)$ for $x\in Q $. We can see that in the Eulerian coordinates,
 $\hat d(x,t)$ is a periodic function in time such that for $t\geq 0$, $\hat
 d(x,t)=\hat d(x,t+1)$ with $T=1$ being the period.

\bigskip

\noindent\textbf{Acknowledgements.}  H. Wu was partially supported
by Natural Science Foundation of Shanghai 10ZR1403800. X. Xu and C.
Liu were partially supported by National Science Foundation grant
NSF-DMS 0707594. The authors would like to thank Professors F.-H.
Lin and M.C. Calderer for their helpful suggestions.

\end{document}